\documentclass[reqno]{article}
\usepackage{amsmath,amsfonts,amssymb,amsthm}

\newtheorem{theorem}{Theorem}
\newtheorem{lemma}[theorem]{Lemma}
\newtheorem{claim}[theorem]{Claim}
\theoremstyle{definition}
\newtheorem{remark}[theorem]{Remark}

\newcommand\eps{\varepsilon}
\renewcommand\le{\leqslant}
\renewcommand\ge{\geqslant}

\newcommand\E{{\mathbb E}}
\renewcommand\Pr{{\mathbb P}}
\newcommand\pto{\overset{\mathrm{p}}{\to}}
\newcommand\dtv[2]{\mathrm{d}_{\mathrm{TV}}\bigl(#1,#2\bigr)}

\newcommand\floor[1]{\lfloor #1 \rfloor}

\newcommand{\sgn}[1]{\mathrm{sgn}({#1})}

\newcommand{\NN}{\mathbb{N}}
\newcommand{\Qp}{{\mathbb Q^+}}
\newcommand{\Rp}{{\mathbb R^+}}
\newcommand{\RR}{\mathbb{R}}

\newcommand\cF{{\mathcal F}}
\newcommand\cD{{\mathcal D}}
\newcommand\cG{{\mathcal G}}
\newcommand\cK{{\mathcal K}}
\newcommand\cL{{\mathcal L}}
\newcommand\cR{{\mathcal R}}
\newcommand\cT{{\mathcal T}}
\newcommand\cU{{\mathcal U}}

\newcommand\fS{{\mathfrak S}}

\newcommand\bv{{\underline v}}
\newcommand\bc{{\underline c}}

\begin{document}
\title{Convergence of Achlioptas processes via differential equations with unique solutions}
\author{Oliver Riordan%
\thanks{Mathematical Institute, University of Oxford, 24--29 St Giles', Oxford OX1 3LB, UK.
E-mail: {\tt riordan@maths.ox.ac.uk}.}
\ and Lutz Warnke%
\thanks{Department of Pure Mathematics and Mathematical Statistics, University of Cambridge, 
Wilberforce Road, Cambridge CB3 0WB, UK.
E-mail: {\tt L.Warnke@dpmms.cam.ac.uk}.}}
\date{June 26, 2013}
\maketitle

\begin{abstract}
In Achlioptas processes, starting from an empty graph, in each step two 
potential edges are chosen uniformly at random, and using some rule one of 
them is selected and added to the evolving graph.  
The evolution of the rescaled size of the largest component in such variations 
of the Erd{\H o}s--R\'enyi random graph process has recently received 
considerable attention, in particular for Bollob\'as's `product rule'. 
In this paper we establish the following result for rules such as the product 
rule: the limit of the rescaled size of the `giant' component exists and is 
continuous provided that a certain system of differential equations has a 
unique solution. 
In fact, our result applies to a very large class of Achlioptas-like processes. 

Our proof relies on a general idea which relates the evolution of 
stochastic processes to an associated system of differential equations. 
Provided that the latter has a unique solution, our approach shows that certain 
discrete quantities converge (after appropriate rescaling) to this solution. 
\end{abstract}

\section{Introduction}
More than 50 years ago Erd\H os and R\'enyi initiated the systematic study of 
the \emph{random graph process}, which is the random sequence of graphs obtained 
by starting with an empty graph on $n$ vertices and then in each step adding a 
new random edge. Already in their seminal 1960 paper~\cite{ErdosRenyi1960} they  
investigated the size of the largest component in great detail. 
Suppressing, as usual, the dependence on $n$, let $L_1(m)$ denote the size of 
the largest component after $m$ steps. Their results imply, for example, that 
there is a continuous function $\rho=\rho^{\mathrm{ER}}:[0,\infty)\to [0,1)$ 
such that for any fixed $t\ge 0$ we have $L_1(\floor{tn})/n \pto \rho(t)$ as 
$n \to \infty$, where $\pto$ denotes convergence in probability. Nowadays the 
evolution of the component structure, in particular the size of the largest 
component, is one of the most studied properties in the theory of random 
graphs, see, e.g., the many references in~\cite{BJR,BR-RG}.  

In an attempt to create processes with potentially different behaviour, 
in 2000 Dimitris Achlioptas suggested certain variants of the classical random 
graph process (inspired by the `power of random choices' paradigm~\cite{ABKU}). 
These also start with an empty graph $G(0)$ on $n$ vertices.  
At each later step $m \ge 1$, two potential edges $e_1$ and $e_2$ are chosen 
independently and uniformly at random from all $\binom{n}{2}$ possible edges 
(or from those edges not present in $G(m-1)$). 
One of these edges is selected according to a rule $\cR$ and added to the graph, 
so $G(m)=G(m-1)\cup \{e\}$ for $e=e_1$ or $e_2$. Processes of this type are 
now known as \emph{Achlioptas processes}; always adding $e=e_1$ gives the 
Erd\H os--R\'enyi random graph process. 

During the last decade the evolution of the largest component in Achlioptas 
processes has received considerable attention. 
In one line of research the location (and existence) of the phase transition 
has been investigated, see e.g.~\cite{BF2001,BohmanKravitz2006,SpencerWormald2007}. 
This is motivated by Dimitris Achlioptas' original question, 
namely, whether the `freedom of choice' in each step can be used to 
substantially delay or accelerate the appearance of the linear size `giant' 
component. 
The above results answer this affirmatively by considering so-called 
`bounded-size' rules, whose decisions only depend on the sizes of the components  
containing the endvertices of $e_1$ and $e_2$, with the restriction that 
all sizes larger than some constant $B$ are treated the same way. 

A more recent direction of research concerns finer details of the phase
transition in Achlioptas processes, see 
e.g.~\cite{BBW2011,BBW2012,JansonSpencer2010,KPS2011}, investigating similarities 
and differences to the well-understood classical random graph process. 
In this context in particular the \emph{product rule} (suggested early on 
by Bollob\'as as the best rule to delay the phase transition) 
has received considerable attention: given two potential edges, it picks the 
one minimizing the product of the sizes of the components of its endvertices. 
Based on extensive simulations, Achlioptas, D'Souza and Spencer conjectured in 
\emph{Science}~\cite{Science2009} that the rescaled size of the largest 
component undergoes a discontinuous phase transition for the product rule, 
i.e., there exists a constant $\delta > 0$ so that $L_1(m)/n$ `jumps' from $o(1)$ 
to at least $\delta$ in $o(n)$ steps. Called \emph{explosive percolation}, 
this phenomenon has been of great interest to physicists, see 
e.g.~\cite{PhysRevE.82.042102,dCDGM2010,DSouzaMitzenmacher2010,PhysRevLett.103.255701,PhysRevE.81.036110,PhysRevLett.103.045701}. 
However, recently it has been rigorously shown 
in~\cite{AAP2011,Science2011} that the simulations were misleading, and that 
the phase transition is actually  continuous for \emph{all} Achlioptas 
processes. 

The discussion above, and much of the physics literature, takes an important 
question for granted: does the scaling limit even exist? More precisely, as 
in~\cite{AAP2011,Science2011} we say that a rule $\cR$ is {\em globally convergent} if 
there exists an increasing function $\rho=\rho^{\cR}:[0,\infty)\to [0,1]$ 
such that for any $t$ at which $\rho$ is continuous we have 
\begin{equation}\label{eq:gc}
 L_1(\floor{tn})/n  \pto \rho(t)
\end{equation}
as $n\to\infty$. The function $\rho=\rho^{\cR}$ is called the {\em scaling limit} of 
(the size of the giant component of) $\cR$.  
Writing $N_{k}(m)$ for the number of vertices of $G(m)$ in components with $k$ 
vertices, we call a rule $\cR$ {\em locally convergent} if there exist functions 
$\rho_k=\rho_k^{\cR}:[0,\infty)\to [0,1]$ such that, for each fixed $k\ge 1$ 
and $t\ge 0$, we have $N_{k}(\floor{tn})/n  \pto \rho_k(t)$ as $n\to\infty$. 
Spencer and Wormald~\cite{SpencerWormald2007} showed that all bounded-size 
rules are locally convergent, and conjectured that they are globally convergent. 
In~\cite{AAP2011} it was shown that global convergence follows from local 
convergence, settling this conjecture. 

For general size rules, the problem of 
establishing convergence (local and hence global) is still open, although 
there are partial results: in~\cite{ESAP} convergence was established up to 
the critical time $t_{\mathrm{b}}$ at which the `susceptibility' (the average 
size of the component containing a random vertex) diverges. 
According to Achlioptas, D'Souza and Spencer~\cite{Science2009}, complex  
rules such as the product rule seem to be `beyond the reach of current mathematical 
techniques', so it is not too surprising that for these no convergence results 
are known beyond $t_{\mathrm{b}}$, i.e., in the later evolution.  
Svante Janson~\cite{Science2011J} also remarks that most likely new 
methods are needed for understanding the detailed behaviour of such rules.

\subsection{Main result}\label{sec:main} 
In this paper we address the convergence question for Achlioptas processes: 
we show that rules such as the product rule are globally convergent 
(for all $t \in [0,\infty)$) 
provided that a certain associated system of differential equations 
(defined in Section~\ref{sec:DEM}) has a unique solution. Our result 
applies to a very large class of Achlioptas-like processes, including 
essentially all Achlioptas processes studied so far. For the definitions of 
$\ell$-vertex rule, merging and well-behaved see Section~\ref{sec:def}.
\begin{theorem}\label{thm:main}
Let $\ell \ge 2$ and let $\cR$ be a merging $\ell$-vertex rule that is well 
behaved.  Suppose the associated system of differential equations given 
by~\eqref{def:DEM:prop}--\eqref{def:DEM:init} has a unique solution 
$(\hat{\rho}_k(t))_{k \ge 1}$. Then $\cR$ is locally and globally convergent. 
In particular, for each fixed $k\ge 1$ and $t\ge 0$, we have 
\begin{equation}\label{eq:lc}
N_{k}(\floor{tn})/n \pto \hat{\rho}_k(t)
\end{equation} 
as $n\to\infty$. The scaling limit $\rho^{\cR}$ is continuous and satisfies 
$\rho^{\cR}(t)=1-\sum_{k\ge 1}\hat{\rho}_k(t)$. 
\end{theorem}
\begin{remark}\label{rem:exists}
We shall show that under the conditions of the theorem, the system of
differential equations has at least one solution (see Lemma~\ref{lem:unique:ss}).
The key assumption of Theorem~\ref{thm:main} is that it does not have more than
one solution.
\end{remark}
\begin{remark}\label{rem:main}
If we only assume uniqueness on an interval $I=[0,t^*]$ or $I=[0,t^*)$, then the conclusions of 
Theorem~\ref{thm:main} hold for any $t \in I$. 
\end{remark}
The merging assumption in Theorem~\ref{thm:main} seems to be necessary (even 
for size rules): in~\cite{PRE2012} examples of `natural' non-merging rules 
are given where simulations strongly suggest that convergence fails. 
All rules for which convergence has been established are merging and well-behaved, 
including the classical Erd\H os--R\'enyi case~\cite{ErdosRenyi1960}, all bounded-size
rules~\cite{SpencerWormald2007} (such as the Bohman--Frieze 
rule~\cite{JansonSpencer2010}) as well as the dCDGM 
rule~\cite{dCDGM2010} and the adjacent edge rule~\cite{DSouzaMitzenmacher2010}. 
In fact, for all such rules there is a $K \ge 1$ such that each $\rho'_k$ 
in~\eqref{def:DEM} can be written as a function of 
$\rho_1, \ldots, \rho_{\max\{k,K\}}$. In this case the form of the differential 
equations~\eqref{def:DEM:prop}--\eqref{def:DEM:init} implies by standard results 
that its solution is unique. So Theorem~\ref{thm:main} generalizes these 
previous convergence results. 

Perhaps the main contribution of this paper is a new approach for proving convergence. 
Previous results in this area apply Wormald's `differential equation  
method'~\cite{Wormald1995DEM,Wormald1999DEM}, which is nowadays widely used in 
probabilistic combinatorics. This shows that under certain conditions, suitable 
sequences of random variables converge to the solution of a system of 
differential equations. The key point is that these conditions imply that 
the differential equations have a unique solution, but are not implied by this. 
By establishing a more direct connection between the random process and the 
differential equations, we only need to assume that the system of differential 
equations has a unique solution. Thus, our method is potentially applicable to a 
much larger class of Achlioptas processes. The general proof idea outlined in 
Section~\ref{sec:proof} might also be useful to establish convergence in other 
stochastic processes. 

In fact, our approach establishes more than convergence:
for each `typical' outcome, it shows that the evolution of suitable random 
variables follows \emph{some} solution of the associated system of 
differential equations. 
Hence our method allows us to transfer properties common to all solutions 
to the random process. 
We demonstrate the usefulness of this feature in Section~\ref{sec:tc}, 
where for Achlioptas processes we narrow the interval in which the giant component emerges. 

Theorem~\ref{thm:main} may be seen as a first step towards resolving the convergence 
question in Achlioptas processes. In particular, further investigation of the 
system of differential equations~\eqref{def:DEM:prop}--\eqref{def:DEM:init} 
associated to the product rule (and other complicated rules) seems to be needed: 
does it have a unique solution? 
When the equations do have a unique solution many questions remain; for example,   
which conditions are needed to establish asymptotic normality as 
in~\cite{Seierstad2009}?

In the next section we define the processes under consideration and state the 
system of differential equations associated to them. In Section~\ref{sec:proof} 
we first outline a general idea for proving convergence in stochastic processes, 
and then use this approach to establish our main result. 
Finally, in Section~\ref{sec:tc} we investigate the emergence of the giant 
component via properties of the different equations.

\section{Preliminaries and notation}\label{sec:def} 
Our core argument will involve considering sequences of points $\omega_n$ 
in different probability spaces. For this reason we indicate the dependence on 
$n$ explicitly in the notation. 
We now recall the relevant definitions from~\cite{AAP2011}.  
Fix $\ell\ge 2$. Each $\ell$-vertex rule $\cR$ yields for each $n$ a random 
sequence $(G_{n,m})_{m\ge 0}$ of graphs with vertex set $[n]=\{1, \ldots, n\}$, 
where $G_{n,0}$ is the empty graph. For each $m\ge 0$ we draw $\ell$ vertices 
$\bv_{n,m+1}=(v_{1},\ldots,v_{\ell})$ from $[n]$ independently and uniformly at 
random, and then obtain $G_{n,m+1}$ by adding a (possibly empty) set of edges 
$E_{n,m+1}$ to $G_{n,m}$, where $\cR$ selects $E_{n,m+1}$ as a subset of all pairs 
between vertices in $\bv_{n,m+1}$. To avoid `trivial' rules we require that 
$E_{n,m+1}\ne\emptyset$ if all $\ell$ vertices in $\bv_{n,m+1}$ are in distinct 
components of $G_{n,m}$ (it would also suffice that the conditional 
probability of this event is bounded away from $0$). 
Formally, we assume the existence of a sample space $\Omega_n$ and a filtration 
$\cF_{n,0}\subseteq \cF_{n,1}\subseteq\cdots$ such that $\bv_{n,m+1}$ is 
$\cF_{n,m+1}$-measurable and independent of $\cF_{n,m}$, and require $E_{n,m+1}$ 
(and hence $G_{n,m+1}$) to be $\cF_{n,m+1}$-measurable. 
For later use we let $\bc_{n,m+1}=(c_{1},\ldots,c_{\ell})$ denote the sizes of 
the components containing the chosen vertices 
$\bv_{n,m+1}=(v_{1},\ldots,v_{\ell})$ in $G_{n,m}$. 
We write $N_{n,k,m}$ for the number of vertices of $G_{n,m}$ in components of 
size $k$, and let $N_{n,\le k,m} = \sum_{1 \le j \le k} N_{n,j,m}$. We define 
$N_{n,\ge k,m}$ in an analogous way. 

For the purposes of this paper these definitions are robust with respect to 
small changes, since our arguments have $o(1)$ elbow room in each step of 
the process. So we may weaken the conditions on $\bv_{n,m+1}$: it suffices if, 
for $m=O(n)$, say, the conditional distribution of $\bv_{n,m+1}$ given 
$\cF_{n,m}$ is close to (at total variation distance $\alpha_n = o(1)$ from) 
the one defined above. 
This includes variations such as picking an $\ell$-tuple of distinct
vertices, or picking (the ends of) $\ell/2$ randomly selected (distinct) 
edges not already present in $G_{n,m}$, see~\cite{AAP2011}. 
Hence we may treat the original examples of Achlioptas as $4$-vertex rules 
where $\cR$ always selects one of the pairs $\{v_{1},v_{2}\},\{v_{3},v_{4}\}$; 
below we call such $\cR$ {\em Achlioptas rules}. 

We say that an $\ell$-vertex rule is {\em merging} if whenever
$C$, $C'$ are distinct components with $|C|,|C'|\ge \eps n$,
then in the next step we have probability at least $\eps^\ell$ of joining
$C$ to $C'$ (this can be slightly weakened, see~\cite{AAP2011}). 
In particular all Achlioptas rules are merging, since  with probability 
at least $\eps^4$ both potential pairs join $C$ to $C'$.

\subsection{Well behaved rules}\label{sec:wbr}
We say that an $\ell$-vertex rule $\cR$ is \emph{well behaved (at infinity)} 
if there are functions $d_{k}:(\NN\cup\{\infty\})^\ell \to \RR$ and $g:\NN\to\NN$
such that the following conditions hold: 
\begin{enumerate}
\item Whenever all vertices $v_j$ are in 
different components we have  
\begin{equation}\label{eq:CE:Nk:change}
\E(N_{n,k,m+1}-N_{n,k,m} \mid \cF_{n,m}, \bv_{n,m+1}) = d_{k}(c_1, \ldots, c_{\ell}) ,
\end{equation}
where $\bc_{n,m+1}=(c_{1},\ldots,c_{\ell})$ lists the sizes of the components 
containing the selected vertices.
\item Suppose there are $I \subseteq [\ell]$ and $S \ge k$ such that all $v_j$ 
with $j \in I$ are in the same component of size $c_j > g(S)$, whereas all other 
vertices are in different components with sizes $c_j \le S$. 
Whenever this holds we have 
\begin{equation}\label{eq:CE:Nk:change:WB}
\E(N_{n,k,m+1}-N_{n,k,m} \mid \cF_{n,m}, \bv_{n,m+1}) = d_{k}(\tilde{c}_1, \ldots, \tilde{c}_\ell) ,
\end{equation}
where $\tilde{c}_j = \infty$ for $j \in I$ and $\tilde{c}_j = c_j$ otherwise.  
\end{enumerate}
In fact, taking $I = \emptyset$ in~\eqref{eq:CE:Nk:change:WB} 
gives~\eqref{eq:CE:Nk:change}, but we note~\eqref{eq:CE:Nk:change} separately 
for clarity. 
As we shall discuss below, these conditions are very mild and hold for 
essentially all Achlioptas processes previously studied, including 
`unbounded rules' such as the sum and product rules. 
All rules which have been considered so far are \emph{size rules}, which 
only use $\bc_{n,m+1}$ to decide which edge(s) are added. 
For these the change of $N_{n,k,m}$ in~\eqref{eq:CE:Nk:change} is deterministic
given $\bc_{n,m+1}$, but considering the conditional expected 
change is slightly more general (we can also allow for small deviations 
in~\eqref{eq:CE:Nk:change} and~\eqref{eq:CE:Nk:change:WB}, but leave 
this to the interested reader). 
Intuitively, the second condition ensures that whenever one component is 
significantly larger than all others, then we can decide which relevant pairs 
are joined \emph{without} knowing its exact size (this fails, for example, if 
the change depends on the parity of $\floor{\log(\max_{j \in [\ell]} c_j)}$). This mild 
assumption holds for a large class of rules; for example, $g(s)=\max\{K,s\}$, 
$g(s)=\max\{B,s\}$, $g(s) = s^2$ and $g(s) = 2s$ suffice for nice rules as 
defined in~\cite{AAP2011}, bounded-size rules, the product rule and the sum 
rule, respectively. 
Note that since $N_{n,k,m}$ always changes by at most $\ell k$ per step, we 
have $|d_{k}(\cdot)| \le \ell k$.

\subsection{An associated system of differential equations}\label{sec:DEM}
Suppose that $\cR$ is a well-behaved $\ell$-vertex rule. 
In the following equations, each $\rho_k(t)$ is a function on $[0,\infty)$ 
satisfying
\begin{equation}\label{def:DEM:prop}
0 \le \rho_k(t) \le 1 
\qquad \text{and} \qquad
0 \le \sum_{k \ge 1} \rho_k(t) \le 1 .
\end{equation}
The system of differential equations associated to $\cR$ is given by 
\begin{equation}\label{def:DEM}
\rho_k'(t) = \sum_{c_1,\ldots,c_\ell \in \NN \cup \{\infty\}} d_{k}(c_1, \ldots, c_{\ell}) \prod_{j \in [\ell]} \rho_{c_j}(t) 
\end{equation}
for all $k \ge 1$, where 
\begin{equation}\label{def:DEM:rho}
\rho(t) = \rho_{\infty}(t) = 1-\sum_{k \ge 1} \rho_k(t) ,
\end{equation}
together with the initial conditions 
\begin{equation}\label{def:DEM:init}
	\rho_k(0) \; = \;	\begin{cases}
		1, & ~~\text{if $k=1$,}\\
		0, & ~~\text{otherwise.}
\end{cases}
\end{equation}
For $t=0$, the derivative in~\eqref{def:DEM} is taken to be the right-derivative. 
Note that for all $t \ge 0$ we have 
$|\rho_k'(t)| \le \max_{\underline{c}}|d_{k}(\underline{c})| \le \ell k$. 

As a basic example, consider the Erd\H os--R\'enyi random graph process, for 
which we have $d_{k}(c_1,c_2) \in \{-2k,-k,0,k\}$. It is not difficult to see 
that in this case~\eqref{def:DEM} simplifies to 
\begin{equation}\label{def:ER}
\rho_k'(t) = - 2k \rho_{k}(t) + k \sum_{c_1+c_2=k} \rho_{c_1}(t) \rho_{c_2}(t) ,
\end{equation}
which is a special case of \emph{Smoluchowski's coagulation equations} in a form
where sol-gel interaction is considered, see e.g.~\cite{Aldous1999,Norris1999} 
and the references therein. 
Here uniqueness follows easily from standard results, since $\rho_k'$ depends 
only on $\rho_1, \ldots, \rho_k$.

\section{Proof of the main result}\label{sec:proof}
We starting by outlining a rather general idea for proving convergence to the 
unique solution of a system of differential equations, which we shall later 
use to establish Theorem~\ref{thm:main}. 
We consider a discrete stochastic process with sample space 
$\Omega_n$ and filtration $\cF_{n,0} \subseteq \cF_{n,1} \subseteq \cdots$.  
For each (discrete) step $m$ we introduce (continuous) time $t=m/s_n$, where 
the scaling satisfies $s_n \to \infty$ as $n \to \infty$. 
Suppose our objective is to find a collection of random variables $X_{n,k,m}$ 
and (continuous) functions $x_k(t)$ together with (deterministic) scaling parameters $S_{n,k}$ 
such that for each fixed $k \ge 1$ and $t \ge 0$, we have 
\begin{equation*}\label{eq:xk:pconv}
X_{n,k,ts_n}/S_{n,k} \pto x_k(t) 
\end{equation*}
as $n \to \infty$, where we ignore the rounding to integers. 
The two main steps of our approach are as follows: 
\begin{enumerate}
	\item Defining the one-step change as $\Delta X_{n,k,m+1}= X_{n,k,m+1}-X_{n,k,m}$, 
we use martingale techniques (the Azuma--Hoeffding inequality together with an 
absolute bound on $|\Delta X_{n,k,m+1}|$) to show that, with probability tending 
to $1$ as $n\to\infty$, the following holds: for each fixed $k$ and all 
$m_1,m_2 \ge 0$ with $m_2-m_1 =O(s_n)$ we have 
\begin{equation}\label{eq:expch:diff}
X_{n,k,m_2}- X_{n,k,m_1}=  \sum_{m_1 \le m < m_2} \E(\Delta X_{n,k,m+1} \mid \cF_{n,m}) + o(S_{n,k}) .
\end{equation}
	\item Suppose we are given a sequence of sample points $\omega_n \in \Omega_n$, 
defined for some infinite set of $n\in \NN$, for which~\eqref{eq:expch:diff} 
and some additional technical conditions hold. Proceeding as in the proof of 
Helley's selection theorem (see, e.g., Theorem~5.8.1 in~\cite{Gut}), we pick a 
subsequence $(\omega_{\tilde{n}})$ such that for each $t \ge 0$ and $k \ge 1$, 
for some limiting value $x_k(t)$ we have 
\begin{equation}\label{eq:xk:conv}
X_{\tilde{n},k,ts_{\tilde{n}}}(\omega_{\tilde{n}})/S_{\tilde{n},k} \to x_k(t)
\end{equation}	
as $\tilde{n} \to \infty$ (here we exploit that each 
$X_{\tilde{n},k,ts_{\tilde{n}}}(\omega_{\tilde{n}})/S_{\tilde{n},k}$ 
satisfies a Lipschitz condition as a function of $t$). 
For this subsequence, we show that for all $t \ge 0$, $\eps > 0$ 
and $k \ge 1$ there exists $\delta > 0$ such that for $\tilde{n}$ large enough 
the following holds: for each $m \ge 0$ satisfying 
$|m-ts_{\tilde{n}}| \le \delta s_{\tilde{n}}$ we have 
\begin{equation}\label{eq:expch:X_k}
\E(\Delta X_{\tilde{n},k,m+1}\mid \cF_{\tilde{n},m})(\omega_{\tilde{n}}) = (f_k(t) \pm \eps) S_{\tilde{n},k}/s_{\tilde{n}} ,
\end{equation}
where $f_k(t)=f_k(t, x_1, x_2, \ldots)$ is a function of the scaling limits 
of the selected subsequence. 
To establish~\eqref{eq:expch:X_k} we combine coupling arguments with 
`typical' properties of the underlying stochastic process. 
\end{enumerate}
Now, using~\eqref{eq:expch:diff}--\eqref{eq:expch:X_k} it is straightforward 
to show that for all $t \ge 0$, $\eps > 0$ and $k \ge 1$ there exists $\delta > 0$ 
such that for all $0 < |h| \le \delta$ with $t+h \ge 0$ we have 
\begin{equation*}
\left| \frac{x_k(t+h)-x_k(t)}{h} - f_k(t) \right| \le \eps,
\end{equation*}
i.e., the $x_k$ satisfy the differential equation 
\begin{equation*}\label{eq:DE}
x'_k(t) =f_k(t, x_1, x_2, \ldots).
\end{equation*}
If the associated system of differential equations has a unique solution, 
then this implies that the limiting functions $x_k(t)$ in~\eqref{eq:xk:conv} 
do \emph{not} depend on the selected subsequence, which establishes the desired 
convergence. 
Finally, let us remark that by comparison with the underlying process we 
can (typically) derive additional properties of the $x_k$; 
it suffices to establish uniqueness of the solution to the system of 
differential equations augmented by these extra restrictions. 

In the remainder we use the above approach to establish 
Theorem~\ref{thm:main}. Aiming at $N_{n,k,tn}/n \pto \rho_k(t)$, we closely follow 
steps one and two in Sections~\ref{sec:proof:main} and~\ref{sec:proof:lemma}, 
respectively, with $X_{n,k,m}=N_{n,k,m}$, $x_k(t)=\rho_k(t)$, $S_{n,k}=n$ and 
$s_n=n$.

\subsection{Proof of Theorem~\ref{thm:main}}\label{sec:proof:main}
Our proof of Theorem~\ref{thm:main} relies on a technical lemma which 
requires some preparation. 
Set $\eta(n) =(\log\log\log n)^{-1}$, say. Let $\cU_n$ denote the event 
that at every step $m$ there is at most one component of size at least $\eta(n) n$. 
Since $\cR$ is merging, by the discussion following Theorem~2 in~\cite{AAP2011} 
we know that $\Pr(\cU_n) \to 1$ as $n \to \infty$. In the rest of the paper the 
particular form of $\eta(n)$ does not matter, only that it tends to $0$ as
$n \to \infty$. 
By Theorem~2 of~\cite{AAP2011}, for any constant $\gamma>0$ there is a constant $K(\gamma)$
such that
\[
 \Pr(\forall m: N_{n,\ge K(\gamma),m} < L_{1,n,m}+\gamma n) \to 1
\]
as $n\to\infty$, where $L_{1,n,m}$ is the number of vertices in the largest component of $G_{n,m}$. 
By a standard argument (considering, say, $\gamma=2^{-i}$ for each $i\in \NN$),
we may allow $\gamma$ to tend to zero at some
rate. More precisely, there exist functions $K(\gamma)$ and 
$\xi(n)$ with $\xi(n) \to 0$ as $n \to \infty$ such that, defining $\cK_n$ as 
the event that for all $m \ge 0$ we have 
\begin{equation}\label{eq:bound:L1:K}
\forall \gamma \ge \xi(n): \ N_{n,\ge K(\gamma),m} < L_{1,n,m} +\gamma n ,
\end{equation} 
we have $\Pr(\cK_n) \to 1$ as $n \to \infty$.

Fix $0 < \lambda < 1/4$, say $\lambda = 1/8$ for concreteness.   
For each $m \ge 0$ set
\[
 \Delta X_{n,k,m+1} = N_{n,k,m+1}-N_{n,k,m}
\]
and 
\[
 Y_{n,k,m+1} = \Delta X_{n,k,m+1} - \E(\Delta X_{n,k,m+1} \mid \cF_{n,m}).
\] 
Set 
\[ 
 Z_{n,k,j} = \sum_{0 \le m < j} Y_{n,k,m+1}.
\] 
Let $\cD_n$ denote the event that for all $1 \le k \le n^\lambda$ and 
$1 \le m_1 \le m_2 \le n^2$ with $m_2-m_1 \le n^{1+\lambda}$ we have 
$|Z_{n,k,m_2}-Z_{n,k,m_1}| < n^{1/2+2\lambda}$. Note that by rearranging 
terms, for all such $k,m_1,m_2$ the event $\cD_n$ implies 
\begin{equation}\label{eq:concentration:N}
N_{n,k,m_2} - N_{n,k,m_1} = \sum_{m_1 \le m < m_2} \E(\Delta X_{n,k,m+1} \mid \cF_{n,m}) \pm n^{1/2+2\lambda} . 
\end{equation}
Since the number of vertices in components of size $k$ 
changes by at most $\ell k$ per step, we have $|\Delta X_{n,k,m+1}| \le \ell k$
and thus 
$|Z_{n,k,m+1}-Z_{n,k,m}|=|Y_{n,k,m+1}| \le 2\ell k$. Furthermore 
$\E(Y_{n,k,m+1} \mid \cF_{n,m}) = 0$, so $(Z_{n,k,j})_{j\ge m_1}$ 
is a martingale. 
Thus, for fixed $k,m_1,m_2$ satisfying the conditions above, 
by the Azuma--Hoeffding inequality we have 
\[
\Pr(|Z_{n,k,m_2}-Z_{n,k,m_1}| \ge n^{1/2+2\lambda}) \le 2e^{-n^{3\lambda}/(8\ell^2k^2)} \le 2e^{-n^{\lambda}/(8\ell^2)} \le n^{-9} 
\]
for $n$ large enough. Taking a union bound (to account for all choices of 
$k,m_1,m_2$) yields $\Pr(\cD_n) \to 1$ as $n \to \infty$. 

Finally, define the `good' event $\cG_n = \cD_n \cap \cK_n \cap \cU_n$; we have
shown that $\Pr (\cG_n) \to 1$ as $n \to \infty$. We are now ready 
to state the main technical lemma. As usual, we ignore the irrelevant rounding 
to integers. 
\begin{lemma}\label{lem:unique:ss} 
Let $\ell \ge 2$ and let $\cR$ be a merging $\ell$-vertex rule that is well 
behaved. Let $(\omega_n)$ with $\omega_n \in \cG_n \subseteq \Omega_n$ be 
defined for an infinite set of $n\in \NN$. Then there exists a subsequence 
$(\omega_{\tilde{n}})$ of $(\omega_n)$ such that for each $t \ge 0$ and $k \ge 1$ we have 
$N_{\tilde{n},k,t \tilde{n}}(\omega_{\tilde{n}})/\tilde{n} \to \rho_k(t)$, 
where the $(\rho_k(t))_{k \ge 1}$ are functions on $\Rp$ satisfying the system 
of differential equations~\eqref{def:DEM:prop}--\eqref{def:DEM:init} associated 
to $\cR$. 
\end{lemma}
Note that Lemma~\ref{lem:unique:ss} implies that the system of differential 
equations~\eqref{def:DEM:prop}--\eqref{def:DEM:init} has at least one solution. 
By comparison with the underlying process we can establish additional properties 
of the $\rho_k(t)$, e.g., that $\rho_{\le k}(t) = \sum_{1 \le j \le k} \rho_j(t)$ 
is monotone decreasing in $t$. 
Before giving the proof of Lemma~\ref{lem:unique:ss}, we first show how it implies 
Theorem~\ref{thm:main}. By Theorem~3 in~\cite{AAP2011} it suffices to 
establish~\eqref{eq:lc}, i.e., local convergence. 
Aiming at a contradiction, suppose there exists $\eps > 0$, ${t_{0}} \ge 0$, 
${k_{0}} \ge 1$ and an infinite set of $\bar{n} \in \NN$ such that 
$|N_{\bar{n},{k_{0}},{t_{0}}\bar{n}}/\bar{n} - \hat{\rho}_{{k_{0}}}({t_{0}})| > \eps$ 
holds with probability at least $\eps$, where $\hat{\rho}_{{k_{0}}}(t)$ 
is given by the (by assumption) unique solution 
to~\eqref{def:DEM:prop}--\eqref{def:DEM:init}. Since $\Pr (\cG_n) \to 1$ as 
$n \to \infty$, this implies (by discarding a finite number of elements in the 
beginning) that there exists an infinite sequence of sample points 
$(\omega_{\bar{n}})$ with 
$\omega_{\bar{n}} \in \cG_{\bar{n}} \subseteq \Omega_{\bar{n}}$ that satisfy  
\begin{equation}\label{eq:Nk:error}
|N_{\bar{n},{k_{0}},{t_{0}}\bar{n}}(\omega_{\bar{n}})/\bar{n} - \hat{\rho}_{{k_{0}}}({t_{0}})| > \eps . 
\end{equation}
Now Lemma~\ref{lem:unique:ss} gives a subsequence $(\omega_{\tilde{n}})$ 
satisfying $N_{\tilde{n},k,t \tilde{n}}(\omega_{\tilde{n}})/\tilde{n} \to \rho_k(t)$
for each $t \ge 0$ and $k \ge 1$, where the $(\rho_k(t))_{k \ge 1}$ 
solve~\eqref{def:DEM:prop}--\eqref{def:DEM:init} on $\Rp$. But 
now~\eqref{eq:Nk:error} implies $\rho_{{k_{0}}}(t_0) \neq \hat{\rho}_{{k_{0}}}(t_0)$, 
contradicting uniqueness.

\subsection{Proof of Lemma~\ref{lem:unique:ss}}\label{sec:proof:lemma}
We start by selecting a `nice' subsequence of $(\omega_n)$, proceeding as in 
the proof of Helley's selection theorem (see, e.g., Theorem~5.8.1 in~\cite{Gut}). 
Define $F_n(k,t) = N_{n,k,tn}(\omega_n)/n$ if $1 \le k \le n$, otherwise set 
$F_n(k,t) = 0$. Clearly, $F_n(k,t) \in [0,1]$. Furthermore, $F_n(1,0)=1$ and 
$F_n(k,0) = 0$ for $k \ge 2$. Let $(q_r)_{r \ge 1}$ be an enumeration of 
$\Qp$. 
A standard diagonal 
argument yields a subsequence $(\omega_{\tilde{n}})$ such that for all 
$(k,q_r) \in \NN \times \Qp$ the value of $F_{\tilde{n}}(k,q_r)$ converges to 
some limit $s_{k,q_r}$. For each $k \in \NN$ we now define 
$\rho_k(q_r) = s_{k,q_r}$ for all $q_r \in \Qp$. Since $N_{n,k,m}$ changes by 
at most $\ell k$ per step, as a function of $t$ each $F_n(k,t)$ is Lipschitz 
on $\Rp$ with constant $\ell k$, so $\rho_k$ has this property on $\Qp$. For 
each $k \in \NN$ we can thus extend $\rho_k$ to a Lipschitz continuous function 
on $\Rp$. Henceforth we always work with the subsequence selected above, but 
write $n$ instead of $\tilde{n}$ for ease of notation. For each
$t \in \Rp$ and $k \in \NN$ we then have
\begin{equation}\label{eq:rhok:conv}
N_{n,k,tn}(\omega_n)/n \to \rho_k(t).
\end{equation}
Turning to some basic properties of the $\rho_k(t)$, by counting vertices we see that
$0 \le \rho_k(t) \le 1$ and 
\begin{equation}\label{eq:rhok:sum}
0 \le \sum_{k \ge 1} \rho_k(t) \le 1 .
\end{equation}
Furthermore, the initial conditions $\rho_1(0)=1$ and $\rho_k(0)=0$ 
for $k \ge 2$ hold. 
To summarize, so far we have established~\eqref{def:DEM:prop} and~\eqref{def:DEM:init}. 

The following \emph{deterministic} statement is the main ingredient in the proof 
of Lemma~\ref{lem:unique:ss}. 
Recall that $\Delta X_{n,k,m+1} = N_{n,k,m+1}-N_{n,k,m}$. 
For brevity, we write $f_k(t)=f_k(t,\rho_1, \rho_2, \ldots)$ for the right 
hand side of~\eqref{def:DEM}. 
\begin{lemma}\label{l:wasclaim}
Let $\omega_n\in \cG_n\subseteq\Omega_n$ be defined for an infinite set of $n\in \NN$,
and suppose that \eqref{eq:rhok:conv} holds.
Then for all $t \ge 0$, $\eps > 0$ and $k \ge 1$ there  exists $0 < \delta \le 1$ 
such that for $n$ large enough the following holds: for each $m \ge 0$ 
satisfying $|m-tn| \le \delta n$ we have 
\begin{equation}\label{eq:exp:change:D_k}
\E(\Delta X_{n,k,m+1} \mid \cF_{n,m})(\omega_n) =  f_k(t) \pm \eps/3 . 
\end{equation}
\end{lemma}
\begin{proof}
Recall that $\omega_n \in \cK_n \cap \cU_n$ satisfies~\eqref{eq:rhok:conv}. 
Given $t\ge 0$, $\eps > 0$ and $k \ge 1$, pick $0<\gamma \le \eps/(50\ell^2 k)$. 
Recall that by definition $\rho(t) = 1-\sum_{k\ge 1}\rho_k(t) \in [0,1]$, see \eqref{def:DEM:rho} and \eqref{eq:rhok:sum}. 
Let $\rho_{< s}(t)= \sum_{1 \le k < s} \rho_k(t)$, which is increasing in $s$
with limit $1-\rho(t)$. 
Choose an integer $S \ge k$ such that $S \ge K(\gamma)$ and 
$\rho_{< S}(t) \ge 1 - \rho(t)-\gamma$, and let
$0<\delta \le \min\{\gamma/(9\ell S^2),1\}$. 
By choice of $S$ we have 
\begin{equation}\label{eq:bound:S:rho}
\rho_{\ge S}(t) = \sum_{k\ge S} \rho_k(t) \le \gamma .
\end{equation}
Consider $m \ge 0$ satisfying $|m-tn| \le \delta n$. 
Since $\eta(n) \le \gamma$ for $n$ large enough, 
using $\omega_n \in \cU_n$ we see that 
\begin{equation}\label{eq:bound:eta:N}
N_{n,\ge \gamma n,m}(\omega_n) > 0 \quad \text{ implies } \quad N_{n,\ge \gamma n,m}(\omega_n) = L_{1,n,m}(\omega_n)  .
\end{equation}
Furthermore, since $\omega_n \in \cK_n$ and $S \ge K(\gamma)$, by~\eqref{eq:bound:L1:K} we have  
$N_{n,\ge S,m}(\omega_n) \le L_{1,n,m}(\omega_n)+ \gamma n$ for $n$ large enough. 
So, by distinguishing whether $L_{1,n,m}(\omega_n)$ is 
larger or smaller than $\gamma n$, we infer
\begin{equation}\label{eq:bound:S:eta}
N_{n,\ge S,m}(\omega_n) - N_{n,\ge \gamma n,m}(\omega_n) \le 2\gamma n .
\end{equation}
We shall now evaluate $\E(\Delta X_{n,k,m+1} \mid \cF_{n,m})(\omega_n)$. 
For this we regard the graph $G_{n,m}(\omega_n)$ as fixed, and the vertices 
$\bv_{n,m+1}=(v_1, \ldots, v_\ell)$ as random. So, in the following all 
probabilities $\Pr(\ast)$ are shorthand for $\Pr(\ast \mid \cF_{n,m})(\omega_n)$. 
Recall the definitions of $\bv_{n,m+1}=(v_1, \ldots, v_\ell)$ and 
$\bc_{n,m+1}=(c_{1},\ldots,c_{\ell})$: the vertices $v_1,\ldots,v_\ell$ are 
chosen independently and uniformly at random from $[n]$, and $c_j$ denotes 
the size of the component in $G_{n,m}(\omega_n)$ containing $v_j$. 
So, for each $s \in [n]$ we have 
\begin{equation*}\label{eq:pr:cj}
\Pr(c_j = s) = N_{n,s,m}(\omega_n)/n .
\end{equation*}
We define $\cT$ as the event that (a) all vertices $v_j$ with $c_j \le S$ are 
in different components, and (b) there are no vertices $v_j$ with 
$S < c_j < \gamma n$. Let $g(\cdot)$ be the function
appearing in the definition \eqref{eq:CE:Nk:change:WB} of well behaved.
Clearly, $\max\{S,g(S)\} < \gamma n$ for $n$ large 
enough. Note that whenever 
$\cT$ holds, by~\eqref{eq:bound:eta:N} all $v_j$ in components of size larger 
than $g(S)$ are in the same component (the largest), so~\eqref{eq:CE:Nk:change} 
or~\eqref{eq:CE:Nk:change:WB} applies, giving 
\begin{equation}\label{eq:exp:change:v}
\E(\Delta X_{n,k,m+1} \mid \cF_{n,m}, \bv_{n,m+1}) = d_{k}(\tilde{c}_1, \ldots, \tilde{c}_\ell) ,
\end{equation}
where $\tilde{c}_j = \infty$ if $c_j \ge \gamma n$, and $\tilde{c}_j = c_j$ 
otherwise. Whether or not $\cT$ holds, the two sides of \eqref{eq:exp:change:v} are bounded by 
$\ell k$. Using~\eqref{eq:bound:S:eta} we see that 
$\Pr(\neg \cT) \le \ell^2 S/n + 2 \ell \gamma$, and so by choice of $\gamma$ we 
have 
\begin{equation*}\label{eq:exp:change:T:err}
2\ell k \cdot \Pr(\neg \cT) \le 2\ell k \cdot (\ell^2 S/n + 2 \ell \gamma) \le \eps/9 
\end{equation*}
for $n$ large enough.
Setting $\fS = [S] \cup \{s \in [n]: s \ge \gamma n\}$, by taking expectations of both sides 
of~\eqref{eq:exp:change:v}, it follows that 
\begin{equation}\label{eq:exp:change:1}
\E(\Delta X_{n,k,m+1} \mid \cF_{n,m})(\omega_n) = \sum_{s_1,\ldots,s_\ell \in \fS} d_{k}(\tilde{s}_1, \ldots, \tilde{s}_\ell) \prod_{j \in [\ell]} \Pr\bigl(c_j=s_j\bigr) \pm \eps/9 ,
\end{equation}
where $\tilde{s}_j = \infty$ if $s_j \ge \gamma n$, and $\tilde{s}_j = s_j$ 
otherwise.

Note that in the estimates above we had plenty of elbow room. So, if the 
conditional distribution of $\bv_{n,m+1}$ is at total variation distance 
$\alpha_n = o(1)$ from the one used above, then a simple coupling argument 
shows that this only adds an additive error of at most $2\ell k \alpha_n$, 
which is negligible for $n$ large enough, 
say at most $\eps/99$. So~\eqref{eq:exp:change:1} is easily seen to still hold in 
such slight variations. 

We define a distribution $Y$ as follows:  
\begin{equation*}\label{eq:def:Y:infty}
\Pr\bigl(Y = \infty\bigr) \; = \;	\begin{cases}
		L_{1,n,m}(\omega_n)/n, & ~~\text{if $L_{1,n,m}(\omega_n) \ge \gamma n$} ,\\
		0, & ~~\text{otherwise}  ,
	\end{cases}
\end{equation*}
and, for all $s \in \NN$,
\begin{equation}\label{eq:def:Y}
\Pr\bigl(Y = s\bigr) \; = \;	\begin{cases}
		N_{n,s,m}(\omega_n)/n, & ~~\text{if $s < \gamma n$} ,\\
		0, & ~~\text{otherwise}  .
	\end{cases}
\end{equation}
Note that by~\eqref{eq:bound:eta:N} this yields a probability distribution. 
Let $Y_1, \ldots, Y_\ell$ be iid with distribution $Y$ and observe that 
$\Pr(Y_j=s)=\Pr(c_j=s)$ for $s \le S < \gamma n$. Since by~\eqref{eq:bound:eta:N} there 
is at most one component of size at least $\gamma n$ (and $\tilde{s}_j = \infty$ 
for $s_j \ge \gamma n$), we see that~\eqref{eq:exp:change:1} gives  
\begin{equation}\label{eq:exp:change:2}
\E(\Delta X_{n,k,m+1} \mid \cF_{n,m})(\omega_n) = \sum_{s_1,\ldots,s_\ell \in [S] \cup \{\infty\}} d_{k}(s_1, \ldots, s_{\ell}) \prod_{j \in [\ell]} \Pr\bigl(Y_j=s_j\bigr) \pm \eps/9 .
\end{equation}
From \eqref{eq:bound:S:eta} and the definition of $Y$ we have $\Pr(S<Y<\infty)\le 2\gamma$.
Since $|d_{k}(\cdot)| \le \ell k$, we can extend the 
sum to all $s_1,\ldots,s_\ell \in \NN \cup \{\infty\}$ at the price of an 
additive error of $2\gamma \ell^2 k$. Since $2\gamma \ell^2 k \le \eps/20$ by 
choice of $\gamma$, this gives 
\begin{equation}\label{eq:exp:change:3}
\E(\Delta X_{n,k,m+1} \mid \cF_{n,m})(\omega_n) = \E(d_{k}(Y_1, \ldots, Y_{\ell})) \pm \eps/6 .
\end{equation}
For $s \le S$ note that $N_{n,s,m}$ changes by at most $\ell s \le \ell S$ in 
each step, so $|m-tn| \le \delta n$ implies 
$|N_{n,s,m}(\omega_n)-N_{n,s,tn}(\omega_n)| \le \ell S \delta n$.  
Hence, using the definition of $\delta$ and~\eqref{eq:rhok:conv},
for $s \le S$ and $n$ large enough we have
\begin{equation}\label{eq:Nsrho}
|N_{n,s,m}(\omega_n)/n - \rho_s(t)| \le \ell S \delta + \gamma/(2S) \le \gamma/S .
\end{equation}
Using this observation we shall 
now show that the right hand side of~\eqref{eq:exp:change:3} is essentially 
determined by the $(\rho_k(t))_{k \ge 1}$; this is key for our approach. 
To this end consider the distribution $Z$ which is defined as follows 
for every $s \in \NN \cup \{\infty\}$: 
\begin{equation}\label{eq:def:Z}
\Pr\bigl(Z = s\bigr) \; = \;	\begin{cases}
		\rho(t), & ~~\text{if $s = \infty$} ,\\
		\rho_s(t), & ~~\text{otherwise}  .
	\end{cases}
\end{equation}
\begin{claim}\label{cl:tvd:Y:Z}
For $n$ large enough we have 
\[ \dtv{Y}{Z} \le 4\gamma  . \]
\end{claim}
\begin{proof}
Recall that the total variation distance is given by 
\begin{equation}\label{eq:dtv:X}
\dtv{Y}{Z} = \frac{1}{2} \sum_{s \in \NN \cup \{\infty\}} \big|\Pr(Y=s)-\Pr(Z=s)\big|  . 
\end{equation}
For $s \le S$, note that~\eqref{eq:Nsrho} readily yields

\[
\sum_{s \in [S]} \big|\Pr(Y=s)-\Pr(Z=s)\big| \le \gamma
\]
for $n$ large enough. 
Next, we consider the summands where $s \in \NN \setminus [S]$. 
Recalling~\eqref{eq:def:Y} and~\eqref{eq:bound:S:eta}, we have
$\Pr\bigl(Y \in \NN \setminus [S]\bigr) \le 2\gamma$.
Similarly, from~\eqref{eq:def:Z} and~\eqref{eq:bound:S:rho} we have 
$\Pr\bigl(Z \in \NN \setminus [S]\bigr) \le \gamma$. Thus
\[
 \sum_{s \in \NN \setminus [S]} \big|\Pr(Y=s)-\Pr(Z=s)\big| \le 3\gamma  . 
\]
Finally, since $Y$ and $Z$ are probability distributions, they differ on 
$s = \infty$ no more than the sum of the differences of the other values, 
i.e., by at most $4\gamma$, and~\eqref{eq:dtv:X} follows.
\end{proof}
Taking $Z_1, \ldots, Z_\ell$ iid with distribution $Z$, using 
Claim~\ref{cl:tvd:Y:Z} the distributions of $(Y_1,\ldots,Y_\ell)$ and $(Z_1,\ldots,Z_\ell)$ can be coupled 
such that they agree with probability at least $1-4\ell\gamma$. 
So, since $|d_{k}(\cdot)| \le \ell k$, in~\eqref{eq:exp:change:3} we may 
replace all occurrences of $Y_j$ by $Z_j$ at the price of an additive error of 
$8\gamma \ell^2 k$. Since $8\gamma \ell^2 k \le \eps/6$ by choice of $\gamma$, 
it follows that
\begin{equation*}\label{eq:exp:change:f_k}
\E(\Delta X_{n,k,m+1} \mid \cF_{n,m})(\omega_n) = \E(d_{k}(Z_1, \ldots, Z_{\ell})) \pm \eps/3  .
\end{equation*}
The first term on the right hand side equals 
$f_k(t)=f_k(t,\rho_1, \rho_2, \ldots)$ by definition of the $Z_j$, 
see~\eqref{def:DEM} and~\eqref{eq:def:Z}. This 
establishes~\eqref{eq:exp:change:D_k} and thus completes the proof of 
Lemma~\ref{l:wasclaim}.
\end{proof}
Finally, with Lemma~\ref{l:wasclaim} in hand, we now complete the proof of 
Lemma~\ref{lem:unique:ss}. 
Given $t\ge 0$, $\eps > 0$ and $k \ge 1$, pick $0 < \delta \le 1$ as given by 
Lemma~\ref{l:wasclaim}. For each $0 < |h| \le \delta$ with $t+h \ge 0$ write 
$m_1,m_2$ for the minimum and maximum of $\{(t+h)n,tn\}$, which satisfy 
$m_1 \ge 0$ and $0 < m_2-m_1 < n^{1+\lambda}$. Recall that $\omega_n \in \cD_n$, 
and note that $k \le n^{\lambda}$ for $n$ large enough. Now,  
using~\eqref{eq:concentration:N} and~\eqref{eq:exp:change:D_k} we see that for 
$n$ large enough 
\begin{equation*}
\begin{split}
N_{n,k,(t+h)n}(\omega_n)-N_{n,k,tn}(\omega_n) &= \sgn{h} \cdot \sum_{m_1 \le m < m_2} \E(\Delta X_{n,k,m+1} \mid \cF_{n,m})(\omega_n) \pm n^{1/2+2\lambda}  \\
 &= hn \cdot (f_k(t) \pm \eps/3 ) \pm n^{1/2+2\lambda} .
\end{split}
\end{equation*}
Rearranging terms, using~\eqref{eq:rhok:conv} and $\lambda < 1/4$ we deduce 
that for $n$ large enough we have  
\begin{equation}\label{eq:rhok:der}
\left| \frac{\rho_k(t+h)-\rho_k(t)}{h} - f_k(t) \right| \le \eps/2 + n^{-1/2+2\lambda}/|h| \le \eps .
\end{equation}
To summarize, for all $t \ge 0$, $\eps > 0$ and $k \ge 1$ there exists 
$\delta > 0$ such that for all $0 < |h| \le \delta$ with $t+h \ge 0$ 
equation~\eqref{eq:rhok:der} holds for $n$ large enough. 
In other words, for $t > 0$ we have $\rho'_k(t) = f_k(t)$, 
which establishes~\eqref{def:DEM}. For $t=0$ we only considered 
$0 < h \le \delta$, so we proved the corresponding statement for the right 
derivative, and the proof of Lemma~\ref{lem:unique:ss} is complete.

\section{Emergence of the giant component}\label{sec:tc}
In this section we demonstrate that our approach may still yield useful 
information in the presence of multiple solutions (to the associated 
system of differential equations): using the emergence of the giant 
component as an example, we show that properties common to all solutions 
of the differential equations usually transfer to the discrete random process. 

We start by briefly recalling the strategy used in the proof of 
Theorem~\ref{thm:main}. Namely, we first defined events $\cG_n$ with 
$\Pr(\cG_n) \to 1$ as $n \to \infty$, and then showed that any sequence 
$(\omega_n)$ of `runs' of an Achlioptas process with $\omega_n \in \cG_n$ 
has a subsequence $(\omega_{\tilde{n}})$ where 
$\left(N_{\tilde{n},k,t\tilde{n}}(\omega_{\tilde{n}})/\tilde{n}\right)_{k \ge 1}$ 
converges to a solution $(\rho_k(t))_{k \ge 1}$ of the associated system of 
differential equations (also with $\rho(t)$ continuous). 
With this in mind, Remark~\ref{rem:main} follows, i.e., for any interval $I$ 
we obtain convergence to the (by assumption) unique solution 
$(\hat{\rho}_k(t))_{k \ge 1}$. 
This is important since it may well be that uniqueness for the system 
of differential equations can be established only up to some point; in 
particular, uniqueness after `gelation', i.e., when 
$\sum_{k \ge 1} \rho_k(t) <1$, seems to be much harder to establish.

In general, we do not know if there is a unique gelation point -- there might 
be a range. However, the following two theorems show that the giant component 
emerges at some point within this range (without assuming any uniqueness).  
\begin{theorem}\label{thm:tclb}
Let $\ell \ge 2$ and let $\cR$ be a merging $\ell$-vertex rule that is 
well behaved. 
Assume that for some $t^* \in [0,\infty)$ every solution 
$(\tilde{\rho}_k(t))_{k \ge 1}$ to the associated system of differential 
equations given by~\eqref{def:DEM:prop}--\eqref{def:DEM:init} satisfies 
$\sum_{k \ge 1} \tilde{\rho}_k(t^*)=1$. 
Then for any $0 \le t \le t^*$ we have $L_1(G_{tn}^{\cR})/n \pto 0$. 
\end{theorem}
\begin{proof}
By monotonicity it suffices to show that $L_1(G_{t^*n}^{\cR})/n \pto 0$. 
Recall that in the proofs we indicate the dependence on $n$ explicitly writing,
for example, $L_{1,n,t^*n}$. Proceeding along the lines of the proof of Theorem~\ref{thm:main}, 
suppose there exists $\delta > 0$ and an infinite set of $\bar{n} \in \NN$ 
with $\Pr(L_{1,\bar{n},t^*\bar{n}}/\bar{n} \ge \delta) \ge \delta$. 
Then, since $\Pr (\cG_n) \to 1$ as $n \to \infty$, there exists an infinite 
sequence of sample points $(\omega_{\bar{n}})$ with 
$\omega_{\bar{n}} \in \cG_{\bar{n}} \subseteq \Omega_{\bar{n}}$ and 
\begin{equation}\label{eq:thm:tclb:L1}
L_{1,\bar{n},t^*\bar{n}}(\omega_{\bar{n}})/\bar{n} \ge \delta . 
\end{equation}
Now Lemma~\ref{lem:unique:ss} gives a subsequence $(\omega_{\tilde{n}})$ 
with $N_{\tilde{n},k,t \tilde{n}}(\omega_{\tilde{n}})/\tilde{n} \to \rho_k(t)$ 
for each $t \ge 0$ and $k \ge 1$, where the $(\rho_k(t))_{k \ge 1}$ 
solve~\eqref{def:DEM:prop}--\eqref{def:DEM:init}. Hence, by 
assumption we have $\sum_{k \ge 1}\rho_k(t^*)=1$, and so for some $K$
we have $\sum_{1 \le k \le K} \rho_k(t^*) \ge 1-\delta/4$. 
Since $L_{1,\tilde{n},m} \le \max\{N_{\tilde{n},\ge K+1,m},K\} \le \tilde{n} - N_{\tilde{n},\le K,m}+K$, 
for $\tilde{n}$ sufficiently large we infer 
\begin{equation}\label{eq:L1:eps}
\begin{split}
L_{1,\tilde{n},t^*\tilde{n}}(\omega_{\tilde{n}})/\tilde{n} &\le 1-N_{\tilde{n},\le K,t^*\tilde{n}}(\omega_{\tilde{n}})/\tilde{n}+K/\tilde{n} \\
& \le 1-  \sum_{1 \le k \le K} \rho_k(t^*) + \delta/4 \le \delta/2 ,  
\end{split}
\end{equation}
contradicting \eqref{eq:thm:tclb:L1}. 
\end{proof}

Our next result gives conditions sufficient to guarantee the emergence of a linear size 
component. The main assumption will be that every solution to the differential equations
has $\rho(t^*)>0$ (i.e., $\sum_{k \ge 1} \rho_k(t^*)<1$). This can be restated as the non-existence
of a solution with $\rho(t^*)=0$; when effectively (as here) imposing the condition
$\rho(t^*)=0$, we may simplify the equations, replacing \eqref{def:DEM} by
\begin{equation}\label{def:DEM:simple}
\rho_k'(t) = \sum_{c_1,\ldots,c_\ell \in \NN} d_{k}(c_1, \ldots, c_{\ell}) \prod_{j \in [\ell]} \rho_{c_j}(t). 
\end{equation}
This generalizes the Smoluchowksi coagulation equations (see e.g.~\cite{Aldous1999,Norris1999})
in a form without sol-gel interaction. The advantage is that it allows
us to drop condition \eqref{eq:CE:Nk:change:WB}.

In the following result $\alpha$ is allowed to depend on $\eps$; this seems  
necessary for rules where the largest component has random size, i.e., which 
are not convergent. 
\begin{theorem}\label{thm:tcub}
Let $\ell \ge 2$ and let $\cR$ be a merging $\ell$-vertex rule that satisfies 
assumption~\eqref{eq:CE:Nk:change}.  
Assume that for $t^* \in [0,\infty)$ every solution $(\tilde{\rho}_k(t))_{k \ge 1}$ 
on $[0,t^*]$ to the associated system of differential equations given 
by~\eqref{def:DEM:prop}, \eqref{def:DEM:init} and \eqref{def:DEM:simple} 
satisfies 
$\sum_{k \ge 1} \tilde{\rho}_k(t^*)< 1$. 
Then for any $t^* \le t < \infty$ and $\eps > 0$ there exist $\alpha,n_0 > 0$ such 
that $\Pr(L_1(G_{tn}^{\cR}) \ge \alpha n) \ge 1-\eps$ for all $n \ge n_0$. 
\end{theorem}
\begin{proof}
By monotonicity it suffices to establish the claim for $t=t^*$. 
Aiming at a contradiction, suppose there exists $\eps > 0$ such that for 
all $\alpha,n_0>0$ we have $\Pr(L_{1,n,t^*n} \le \alpha n) \ge \eps$ 
for some $n \ge n_0$. 
It follows as usual that there is $\hat{\alpha}(n) \to 0$ as $n \to \infty$ 
and an infinite set of $\bar{n} \in \NN$ such that 
$\Pr(L_{1,\bar{n},t^*\bar{n}} \le \hat{\alpha}(\bar{n}) \bar{n}) \ge \eps$. 
Define $\cL_n$ as the event that $L_{1,n,t^*n} \le \hat{\alpha}(n) n$. 
Since $\Pr (\cG_n) \to 1$ as $n \to \infty$ there exists an infinite 
sequence of sample points $(\omega_{\bar{n}})$ with 
$\omega_{\bar{n}} \in \cG_{\bar{n}} \cap \cL_{\bar{n}} \subseteq \Omega_{\bar{n}}$, 
for which we now prove the following variant of Lemma~\ref{lem:unique:ss}. 
\begin{claim}\label{cl:unique:mod}
There is a subsequence $(\omega_{\tilde{n}})$ of $(\omega_{\bar{n}})$ such that for each 
$0 \le t \le t^*$ and $k \ge 1$ we have 
\begin{equation}\label{eq:cl:unique:mod:Nk}
N_{\tilde{n},k,t \tilde{n}}(\omega_{\tilde{n}})/\tilde{n} \to \rho_k(t),
\end{equation}
where the $(\rho_k(t))_{k \ge 1}$ are functions satisfying the 
system of differential equations~\eqref{def:DEM:prop}, \eqref{def:DEM:init}, 
\eqref{def:DEM:simple} on $[0,t^*]$.  
\end{claim}
\begin{proof}
Defining $d_k(c_1, \ldots, c_\ell)=0$ if any argument is infinite, note that 
\eqref{def:DEM:simple} equals \eqref{def:DEM}. So, in view 
of Section~\ref{sec:proof:lemma}, it suffices to prove Lemma~\ref{l:wasclaim} for 
$0 \le t \le t^*$. We closely follow the original argument, only 
changing some minor details (we also write $n$ instead of $\tilde{n}$ for ease of 
notation). When selecting the parameters $\gamma$, $S$, $\delta$ 
we use $S \ge K(\gamma/9)$ instead of $S \ge K(\gamma)$. 
Observe that $N_{n,\ge S,m}$ increases by at most $\ell S$ in each step, so that 
$N_{n,\ge S,(t^*+\delta)n} \le N_{n,\ge S,t^*n} + \ell \delta S n$, where 
$\ell \delta S \le \gamma/9$ by choice of $\delta$. 
Since $\omega_n \in \cK_n \cap \cL_n$ and $S \ge K(\gamma/9)$, we have 
\begin{equation}\label{eq:NgeS}
 N_{n,\ge S,t^*n}(\omega_n) \le L_{1,n,t^*n}(\omega_n)+ \gamma n/9 \le (\hat{\alpha}(n)+\gamma/9) n \le \gamma n/3,
\end{equation}
for $n$ large enough. For $t\le t^*$, combining these estimates with monotonicity, we
deduce that for $n$ sufficiently large we have
\[
 L_{1,n,m}(\omega_n) \le L_{1,n,(t^*+\delta)n}(\omega_n) \le \max\{N_{n,\ge S,(t^*+\delta)n}(\omega_n) ,S\} \le \gamma n/2 
\]
for every $m \ge 0$ with $|m-tn| \le \delta$.

When establishing \eqref{eq:exp:change:v} the assumption~\eqref{eq:CE:Nk:change} 
thus always applies (whenever the event $\cT$ holds all vertices are in different 
components and satisfy $c_j \le S$). 
Consequently \eqref{eq:exp:change:1} holds, since $d_k(c_1, \ldots, c_\ell)=0$ 
if $\infty \in \{c_1, \ldots, c_{\ell}\}$. Now the remainder of the argument 
leading to Lemma~\ref{l:wasclaim} is unchanged, which, as discussed, completes 
the proof of Claim~\ref{cl:unique:mod}. 
\end{proof}
Now consider a subsequence $(\omega_{\tilde{n}})$ with the properties 
guaranteed by Claim~\ref{cl:unique:mod}. From \eqref{eq:NgeS}, for $n$ large we have
\[
 N_{\tilde{n},\le S,t^*\tilde{n}}(\omega_{\tilde{n}}) \ge
 n-  N_{\tilde{n},\ge S,t^*\tilde{n}}(\omega_{\tilde{n}}) \ge (1-\gamma/3) n,
\]
so from \eqref{eq:cl:unique:mod:Nk} it follows that
\[
 \sum_{k\ge 1}\rho_k(t^*) \ge \sum_{1 \le k \le S} \rho_k(t^*) \ge 1-\gamma.
\]
Since we could choose the constant $\gamma$ arbitrarily small, we have $\sum_{k\ge 1}\rho_k(t^*)\ge 1$.
Since $(\rho_k(t))_{k \ge 1}$ is a solution to \eqref{def:DEM:prop}, \eqref{def:DEM:init} and
\eqref{def:DEM:simple} on $[0,t^*]$, this contradicts the assumptions of the theorem.
\end{proof}

\begin{remark}
We may replace \eqref{def:DEM:simple} by \eqref{def:DEM} in the assumptions of Theorem~\ref{thm:tcub}. 
Indeed, using \eqref{eq:rhok:sum}, the $(\rho_k(t))_{k \ge 1}$ constructed above satisfy $\sum_{k \ge 1} \rho_k(t)=1$ for $0 \le t \le t^*$. 
Thus they also solve \eqref{def:DEM}, since $\rho(t)=0$ for $0 \le t \le t^*$. 
\end{remark}

\begin{remark}\label{rem:tcub}
Theorem~\ref{thm:tcub} also holds without the merging assumption; we outline
the minor modifications needed to the proof.
Using Remark~9 in~\cite{AAP2011}, we
replace `at most one component' by `at most $\ell-1$ components' in the 
definition of $\cU_n$, and replace $L_{1,n,m}$ by $L_{n,m}$ in the 
definition of $\cK_n$, where $L_{n,m}$ denotes the sum of the sizes of the 
$\ell-1$ largest components. 
Now, thinking of all `infinity' terms as the probability of being in one of 
the $\ell-1$ largest components, using $L_{n,m} \le \ell \cdot L_{1,n,m}$ it 
is not difficult to push the argument through; we omit the details.
(Since the $\ell-1$ largest components may not all be large, in 
the argument leading to Claim~\ref{cl:unique:mod} it may be convenient to 
use $L_{\gamma,n,m}$, defined as the sum of the $\ell-1$ largest 
components with size at least $\gamma n$.) 
\end{remark}

It might be surprising that the sol-gel interaction and 
condition~\eqref{eq:CE:Nk:change:WB} are used in Theorem~\ref{thm:tclb}, but 
not Theorem~\ref{thm:tcub} (rather than the other way round). 
Here the explanation is that our proofs proceed by contradiction, 
showing the existence of a gelating solution in case of Theorem~\ref{thm:tclb}, 
and a non-gelating solution in case of Theorem~\ref{thm:tcub}. 
Nevertheless, since condition~\eqref{eq:CE:Nk:change:WB} essentially ensures 
that the giant component, once it emerges, evolves in a regular way, it may 
well not be needed in Theorem~\ref{thm:tclb}.

\end{document}